\documentclass[12pt]{article}
\usepackage{amsmath,amssymb,setspace,amsthm,amsfonts,diagrams,graphicx,float}

\def\Hom{{\rm Hom}}

\def\K{\mathbb{K}} 
\def\phi{\varphi}
\def\g{\mathfrak g}
\def\F{\mathbb F}

\makeatletter
\let\@@pmod\pmod
\DeclareRobustCommand{\pmod}{\@ifstar\@pmods\@@pmod}
\def\@pmods#1{\mkern4mu({\operator@font mod}\mkern 6mu#1)}
\makeatother

\newtheorem{theorem}{Theorem}[section]

\newtheorem{corollary}[theorem]{Corollary} \theoremstyle{remark}

\providecommand{\keywords}[1]{\noindent{Keywords:} #1}
\providecommand{\classify}[1]{\noindent{Mathematics Subject Classification:} #1}

\title{Restricted One-dimensional Central Extensions of Restricted
  Simple Lie Algebras}

\author{Tyler J. Evans\footnote{This paper was written during the first
    author's sabbatical leave at E\" otv\" os Lor\' and
  University in Budapest, Hungary.} \\Department of Mathematics \\Humboldt State
  University \\Arcata, CA 95521 USA \\evans@humboldt.edu \and Alice
  Fialowski \\Institutes of Mathematics \\University of P\' ecs and E\" otv\" os Lor\' and
  University \\H-7624 P\'ecs and H-1117 Budapest, Hungary\\ fialowsk@cs.elte.hu}

\date{}

\begin{document}
\maketitle


\classify{17B56, 17B50}

\section{Introduction}
In \cite{EvansFialowskiPenkava}, the authors give an explicit
description of the cocycles parameterizing the space of restricted
one-dimensional central extensions of the Witt algebra $W(1)=W(1,{\bf
  1})$ defined over fields of characteristic $p\ge 5$. The Witt
algebra is a finite dimensional restricted simple Lie algebra, and
such algebras have been completely classified for primes $p\ge 5$
\cite{StradeBook}. In this paper, we study the restricted
one-dimensional central extensions of an arbitrary finite dimensional
restricted simple Lie algebra for $p\ge 5$.

The one-dimensional central extensions of a Lie algebra $\g$ defined
over a field $\F$ are classified by the Lie algebra cohomology group
$H^2(\g)=H^2(\g,\F)$ where $\F$ is taken as a trivial $\g$-module. The
restricted one-dimensional central extensions of a restricted Lie
algebra $\g$ over $\F$ are likewise classified by the restricted Lie
algebra cohomology group $H^2_*(\g)=H^2_*(\g,\F)$. We refer the reader
to \cite{EvansFialowskiPenkava} and \cite{EvansFuchs2} for
descriptions of the complexes used to compute the ordinary and
restricted Lie algebra cohomology groups as well as a review of the
correspondence between one-dimensional central extensions (restricted
central extensions) and the second cohomology (restricted cohomology)
group. In particular, we adopt the notation and terminology in
\cite{EvansFialowskiPenkava}.

The dimensions of the ordinary cohomology groups $H^2(\g)$ for finite
dimensional simple restricted Lie algebras are known
\cite{BlockExt,Chiu,FarnsteinerCentral,FarnsteinerExt,PremetStradeMelikian}. Following
the technique used in \cite{VivianiResDef}, we use these results along
with Hochschild's six term exact sequence relating the first two
ordinary and restricted cohomology groups to analyze the restricted
cohomology group $H^2_*(\g)$. Our theorem states that if $\g$ is a
restricted simple Lie algebra, this sequence reduces to a short exact
sequence relating $H^2(\g)$ and $H^2_*(\g)$. In the case that
$H^2(\g)=0$, we explicitly describe the cocycles spanning
$H^2_*(\g)$. If $H^2(\g)\ne 0$, we give a procedure for describing a
basis for $H^2_*(\g)$.

The paper is organized as follows. Section 2 gives an overview of the
classification of finite dimensional simple restricted Lie algebras $\g$
defined over fields of characteristic $p\ge 5$. Section 3 contains the
statement and proof of the theorem relating $H^2(\g)$ and $H^2_*(\g)$
as well as an explicit description of the cocycles spanning
$H^2_*(\g)$ when $H^2(\g)=0$. Section 4 outlines a procedure for
describing a basis for $H^2_*(\g)$ in the case where $H^2(\g)\ne 0$.

\paragraph {\bf Acknowledgements.}
The present work is dedicated to the 650th anniversary of the
foundation of the University of P\' ecs, Hungary. 

The authors are
grateful for the suggestions of the referee as they greatly improved
the exposition.

\section{Restricted Simple Lie Algebras}

Finite dimensional simple Lie algebras over fields of characteristic
zero were classified more than a century ago in the work of Killing
and Cartan. The well known classification theorem states that in
characteristic zero, any simple Lie algebra is isomorphic to one of
the linear Lie algebras $A_l, B_l, C_l$ or $D_l$ ($l\ge 1$), or one of
the exceptional Lie algebras $E_6,E_7,E_8,F_4$ or $G_2$. Levi's
theorem implies that in characteristic 0, all extensions of a
semi-simple Lie algebra split, and hence all one-dimensional central
extensions of such an algebra are trivial.

The classification of modular simple Lie algebras over fields of
characteristic $p\ge 5$ was completed more recently in the work by Block,
Wilson, Premet and Strade in multiple papers spanning decades of
research. We refer the reader to \cite{StradeBook} for detailed
account of this work. The classification states that simple Lie
algebras of characteristic $p\ge 5$ fall into one of three types: {\it
  classical} Lie algebras, algebras of {\it Cartan type} and {\it
  Melikian algebras} (Melikian algebras are defined only for $p=5$).

The classical type simple Lie algebras are constructed by using a
Chevalley basis to construct a Lie algebra $L$ and
tensoring over $\F$ to yield a Lie algebra $L_\F$ over $\F$. The
algebras $L_\F$ are all restrictable, and they are all simple unless
$L_\F\simeq A_l$ where $p|(l+1)$. In this case $L_\F$ has a one
dimensional center $C$ and the quotient $\mathfrak {psl}(l+1)=L_\F/C$
is simple \cite{BlockExt,StradeBook}. Thus the simple restricted
algebras of classical type are:
\[A_l (p\not | (l+1)), \mathfrak{psl}(l+1) (p| (l+1)),
B_l,C_l,D_l,G_2, F_4,E_6,E_7,E_8.\] Block shows in \cite{BlockExt}
(Theorem 3.1) that if $\g$ is a simple modular Lie algebra of
classical type and $\g\not\simeq \mathfrak{psl}(l+1)$ where $p|(
l+1)$, then $\g$ has no non-trivial central extensions at all so that
in particular $H^2(\g)=0$. Moreover, the same theorem implies that if
$p|(l+1)$, any one-dimensional central extension of
$\mathfrak{psl}(l+1)$ is equivalent to the trivial one-dimensional
central extension so that $H^2(\mathfrak{psl}(l+1))=0$ as well.

Algebras of Cartan type were constructed by Kostrikin, Shafarevich
and Wilson and are divided into four families, called {\it
  Witt-Jacobson ($W(n,{\bf m})$), Special ($S(n,{\bf m}))$, Hamiltonian
  ($H(n,{\bf m})$)} and {\it Contact ($K(n,{\bf m})$) algebras}
\cite{StradeBook}. Unlike the classical type algebras, not all algebras
of Cartan type are restrictable. The restrictable simple
Lie algebras of Cartan type are $W(n)=W(n,{\bf 1}), S(n)=S(n,{\bf 1}),
H(n)=H(n,{\bf 1})$ and $K(n)=K(n,{\bf 1})$
\cite{StradeBook}. We refer the reader to \cite{VivianiResDef} for
explicit descriptions of these algebras. Some of these algebras have
non-trivial one-dimensional central extensions as Lie algebras
\cite{Chiu,FarnsteinerCentral,FarnsteinerExt}.

If $p=5$, there is one more family of simple modular Lie algebras
called {\it Melikian algebras} that were first introduced in
\cite{Melikian}. The only restrictable algebra in this family is the
algebra $M=M(1,1)$, and $M$ has no non-trivial one-dimensional central
extensions as a Lie algebra \cite{PremetStradeMelikian}. We refer the
reader to \cite{VivianiDefM} for an explicit description of this
algebra.

\section{Restricted One-dimensional Central Extensions}

If $\g$ is a restricted Lie algebra, there is a six-term exact
sequence in \cite{H} that relates the
ordinary and restricted cohomology groups in degrees one and two.
\begin{diagram}[LaTeXeqno]
  \label{sixterm}
  0 &\rTo &H^1_*(\g,M)&\rTo &H^1(\g,M)&\rTo&\Hom_{\mbox{\rm
      fr}}(\g,M^\g) & \rTo \\
  & \rTo & H^2_*(\g,M)&\rTo &H^2(\g,M)&\rTo&\Hom_{\mbox{\rm
      fr}}(\g,H^1(\g,M)) &
\end{diagram}
Here we use the same notation as in \cite{VivianiResDef} and write
$\Hom_{\mbox{\rm fr}}(V,W)$ for the set of {\it Frobenius
  homomorphisms} from the $\F$-vector space $V$ to the $\F$-vector
space $W$. That is
\[ \Hom_{\mbox{\rm fr}}(V,W) = \{ f:V\to W\ |\ f(\alpha x+ \beta y) =
\alpha^p f(x) + \beta^p f(y)\}\] for all $\alpha,\beta\in\F$ and
$x,y\in V$. We remark that $f\in \Hom_{\mbox{\rm fr}}(\g,\F)$ if and
only if $f$ is $p$-semilinear as defined in \cite{EvansFialowskiPenkava}.

\begin{theorem}
  If $\g$ is a restricted simple Lie algebra defined over a field $\F$
  of characteristic $p>0$, then as vector spaces over $\F$,
  \[H^2_*(\g)=H^2(\g)\oplus \Hom_{\mbox{\rm fr}}(\g,\F),\] so that
  \[\dim H^2_*(\g) = \dim H^2(\g) + \dim \g.\]
\end{theorem}

\begin{proof}
  With trivial coefficients, the ordinary coboundary map $\delta^1:C^1(\g)\to
  C^2(\g)$ reduces to
  \[\delta^1\phi(g_1,g_2)=\phi([g_1,g_2]).\]
  Since $\g$ is simple, this shows $H^1(\g)=0$, and it follows from
  (\ref{sixterm}) that $H^1_*(\g)=0$ as well. Moreover
  $\Hom_{\mbox{\rm fr}}(\g,H^1(\g,M))=\Hom_{\mbox{\rm fr}}(\g,0)=0$
  and $\F^\g=\F$. Therefore the sequence (\ref{sixterm}) reduces to a
  short exact sequence
  \begin{diagram}
    0 &\rTo&\Hom_{\mbox{\rm fr}}(\g,\F) & \rTo & H^2_*(\g)&\rTo
    &H^2(\g)&\rTo& 0
  \end{diagram}
\end{proof}

\begin{corollary}
  \label{resonly}
  If $\dim\g=n$, and $\{x_1,\dots, x_n\}$
  is a basis for $\g$, then there is a $n$-dimensional subspace of $H^2_*(\g)$ spanned by
  restricted cohomology classes represented by restricted cocycles
  $(0,\omega_i)$ ($i=1,\dots ,n$) where $\omega_i:\g\to\F$ is
  defined by
  \[\omega_i\left ( \sum_{j=1}^n \alpha_j x_j\right
  )=\alpha_j^p.\] Moreover, if $E_i$ denotes the one-dimensional
  restricted central extension of $\g$ determined by the cohomology
  class of the cocycle $(0,\omega_i)$, then $E_i=\g\oplus \F c$ as a
  $\F$-vector space, and we have for all $1\le
  j,k\le n$,
  \begin{align}
    \begin{split}
      \label{extensions}
      [x_j,x_k] & =[x_j,x_k]_\g;\\
      [x_j, c] & = 0;\\
      e_{j}^{[p]} & = x_j^{[p]_\g}+ \delta_{i,j} c;\\
      c^{[p]} & = 0,
    \end{split}
  \end{align}
  where $[x_j,x_k]_\g$ and $x_j^{[p]_\g}$ denote the Lie bracket and
  $[p]$-operation in $\g$ respectively, and $\delta$ denotes the
  Kronecker delta-function.
\end{corollary}

\begin{proof}
  Easy computations show the maps $\omega_i$ are a basis for
  $\Hom_{\mbox{\rm fr}}(\g,\F)$, $(0,\omega_i)$ is a cocycle for all
  $i$, and the corresponding cohomology classes are
  linearly independent in $H^2_*(\g)$. The Lie bracket and
  $[p]$-operation in $E_i$ are given by the equations (4) in \cite{EvansFialowskiPenkava}.
\end{proof}

\section{Case By Case Study}
In this section, all restricted Lie algebras are defined over fields
of characteristic $p\ge 5$. Corollary \ref{resonly} implies if $H^2(\g)=0$, then the extensions in
(\ref{extensions}) are the only restricted one-dimensional central
extensions of a restricted simple Lie algebra $\g$. Table
\ref{table:zero} lists the restricted simple Lie algebras $\g$ for
which $H^2(\g)=0$ along with the dimensions of these algebras
(i.e. the dimensions of the space of restricted one-dimensional
central extensions). This table contains all the simple restricted Lie
algebras of classical type \cite{BlockExt}, the Witt algebras $W(n)$
for $n>1$ \cite{FarnsteinerExt}, the Contact algebras $K(n)$ for
$n+3\not\equiv 0\pmod p$ \cite{FarnsteinerExt} and the Melikian
algebra $M$ \cite{PremetStradeMelikian}. For each of these algebras,
we have explicit descriptions (\ref{extensions}) of all restricted
one-dimensional central extensions.
\begin{table}[H]
  \resizebox{.9\textwidth}{!}{\begin{minipage}{\textwidth}
      \begin{tabular}{lr}
        & \\
        $\g$ & $\dim\g=\dim H^2_*(\g)$\\
        \hline \hline
        $A_l$ ($l\ge 1$, $p\not | (l+1)$) & $(l+1)^2-1$ \\
        $\mathfrak{psl}(l+1)$ ($l\ge 1$, $p | (l+1)$) & $(l+1)^2-2$ \\
        $B_l$  ($l\ge 2$) &  $2l^2+l$\\
        $C_l$ ($l\ge 3$)   & $2l^2+l$ \\ 
        $D_l$ ($l\ge 4$)   &  $2l^2-l$\\
        $G_2$  &  14 \\
        $F_4$  & 52 \\
        $E_6$  & 78 \\
        $E_7$  & 133 \\
        $E_8$  &  248\\
        $W(n)$ ($n>1$) & $np^n$\\
        $K(n)$ ($n+3\not\equiv 0\pmod p$) & $p^n$\\
        $M$ & 125\\
      \end{tabular}
      \caption[Table caption text]{$\dim H^2_*(\g)$ for $\g$ for which
        $H^2(\g)=0$}
      \label{table:zero}
    \end{minipage} }
\end{table}
Table \ref{table:nonzero} lists the restricted simple Lie algebras
$\g$ for which $H^2(\g)\ne 0$ along with the dimensions of these
algebras, the dimensions of the ordinary cohomology groups and the
dimensions of the space of restricted one-dimensional central
extensions. This table includes the Witt algebra $W(1)$
\cite{EvansFialowskiPenkava,FarnsteinerExt}, the Special algebras
$S(n)$ \cite{Chiu}, Hamiltonian algebras $H(n)$ \cite{Chiu}, and the
Contact algebras $K(n)$ for $n+3\equiv 0\pmod p$
\cite{FarnsteinerExt}.

\begin{table}[H]
  \resizebox{.9\textwidth}{!}{\begin{minipage}{\textwidth}
      \begin{tabular}{lrrr}
        & \\
        $\g$ & $\dim\g$ & $\dim H^2(\g)$ & $\dim H^2_*(\g)$\\
        \hline \hline
        $W(1)$ & $p$ & 1 & $p+1$ \\
        $S(3)$ & $2p^3-2$ & 3 & $2p^3+1$\\
        $S(n)$ ($n>3$)& $(n-1)(p^n-1)$ & $n(n-1)/2$ & $(n-1)(2p^n+n-2)/2$\\
        $H(n)$ ($n+4\equiv 0\pmod p$) & $p^n-2$ & $n+2$ & $p^n+n$\\
        $H(n)$ ($n+4\not\equiv 0\pmod p$) & $p^n-2$ & $n+1$& $p^n+n-1$\\
        $K(n)$ ($n+3\equiv 0\pmod p$) & $p^n-1$ & $n+1$& $p^n+n$
      \end{tabular}
      \caption[Table caption text]{$\dim H^2_*(\g)$ for $\g$ for which
        $H^2(\g)\ne 0$}
      \label{table:nonzero}
    \end{minipage} }
\end{table}
For $W(1)$, the non-trivial ordinary cohomology class determines a
restricted one-dimensional central extension $E=W(1)\oplus \K c$ where
we have for all $-1\le j,k\le p-2$,
\begin{align*}
  [e_j,e_k] & =(k-j) e_{j+k}+\frac{j(j^2-4)}{3}\delta_{0,j+k} c;\\
  [e_j, c] & = 0;\\
  e_{j}^{[p]} & = \delta_{0,j}e_0;\\
  c^{[p]} & = 0
\end{align*}
and $\{e_{-1},\dots,e_{p-2}\}$ is a basis for $W(1)$
\cite{EvansFialowskiPenkava}.  \medskip

Finally, we remark that in \cite{VivianiResDef}, the map
\begin{diagram}
  H^2(\g,M)&\rTo^\Delta&\Hom_{\mbox{\rm fr}}(\g,H^1(\g,M))
\end{diagram}
from the exact sequence (\ref{sixterm}) is given explicitly. In the
case of trivial coefficients, this map $\Delta$ is given by
\[\Delta_\phi(g)\cdot h = \phi(g,h^{[p]})-\phi([g,
\underbrace{h,\cdots,h}_{p-1}], h)\] where $\phi\in C^2(\g)$ and
$g,h\in\g$.  Since $\g$ is simple, $H^1(\g)=0$ and hence $\Delta$ is
the zero map so that
\[\phi(g,h^{[p]})-\phi([g,
\underbrace{h,\cdots,h}_{p-1}], h)=0\] for all $\phi\in C^2(\g)$ and
$g,h\in\g$. Therefore if $\phi\in C^2(\g)$ is a cocycle and $\omega:\g
\to\F$ has the $*$-property with respect to $\phi$ (as defined in
\cite{EvansFialowskiPenkava}), then $(\phi,\omega)\in C^2_*(\g)$ is a
restricted cocycle. Moreover, given $\phi\in C^2(\g)$, we can set the
value of $\omega$ on every basis vector for $\g$ to be 0 and use
equation (1) in \cite{EvansFialowskiPenkava} to define $\omega$ to
have the $*$-property with respect to $\phi$. If $\phi_1$ and $\phi_2$
are linearly independent cocycles in $C^2(\g)$, the resulting
restricted cocycles $(\phi_1,\omega_1)$ and $(\phi_2,\omega_2)$ are
linearly independent in $C^2_*(\g)$. This gives us a canonical process
of producing a basis for $H^2_*(\g)$ from a basis for $H^2(\g)$ and
the cocycles $(0,\omega_i)$ in Corollary \ref{resonly}. This process
was used to produce the basis for $H^2_*(W(1))$ in
\cite{EvansFialowskiPenkava}.

\bibliography{references}{} \bibliographystyle{plain}
\end{document}